\documentclass[english]{amsart}

\usepackage[shortlabels]{enumitem}
\setlist[enumerate]{label=\rm{(\arabic*)}}

\usepackage{mathrsfs}
\usepackage{mathtools}
\usepackage{graphicx}
\usepackage{array}
\usepackage{amsmath}
\usepackage{amssymb}
\usepackage{amsthm}
\usepackage{amsfonts}
\usepackage[hyperindex=true,colorlinks=true,linkcolor=blue,citecolor=biblio,urlcolor=biblio]{hyperref}
\usepackage{cleveref}

\usepackage{tikz}
\usetikzlibrary{matrix,cd,arrows.meta,fit,calc,positioning}
\tikzcdset{
arrow style=tikz,
diagrams={>={Stealth[scale=0.8]}}
}

\theoremstyle{plain}
\newtheorem{thm}{Theorem}[section]
\newtheorem*{thm*}{Theorem}
\newtheorem{cor}[thm]{Corollary}
\newtheorem{prop}[thm]{Proposition}
\newtheorem*{prop*}{Proposition}
\newtheorem{lemma}[thm]{Lemma}

\theoremstyle{definition}
\newtheorem*{notation}{Notation}
\newtheorem{ex}[thm]{Example}
\newtheorem{defi}[thm]{Definition}
\newtheorem{rem}[thm]{Remark}

\numberwithin{equation}{subsection}

\newcounter{SExactes}
\newcounter{Present}

\DeclareMathOperator{\codim}{codim}

\DeclareMathOperator{\depth}{depth}

\DeclareMathOperator{\Fitt}{Fitt}

\DeclareMathOperator{\Proj}{Proj}
\DeclareMathOperator{\Sym}{Sym}

\renewcommand{\k}{\mathrm{k}}
\renewcommand{\P}{\mathbb{P}}
\newcommand{\mcP}{\mathcal{P}}

\newcommand{\bbB}{\mathbb{B}}
\renewcommand{\O}{\mathcal{O}}
\DeclareMathOperator{\tnS}{S}
\DeclareMathOperator{\tnV}{V}

\newcommand{\V}{\mathbb{V}}
\newcommand{\bbW}{\mathbb{W}}
\newcommand{\IW}{\I_{\mathbb{W}}}

\newcommand{\I}{\mathcal{I}}
\newcommand{\J}{\mathcal{I}_{\PI}}
\newcommand{\IZ}{{\mathcal{I}_{Z}}}

\newcommand{\PI}{\mathbb{X}}

\newcommand{\IPIK}{\I_{\PI,\mfK}}

\newcommand{\PnX}{\mathbb{P}_X^n}

\newcommand{\Pnk}{\P^{n}}

\newcommand{\E}{\mathcal{E}}
\newcommand{\F}{\mathcal{F}}
\newcommand{\K}{\mathcal{K}}
\newcommand{\mfK}{\mathbb{K}}

\newcommand{\mfQ}{\mathcal{Q}}

\newcommand{\p}{p^*}
\newcommand{\Ext}{\mathcal{E}\textnormal{xt}}

\renewcommand{\L}{\mathcal{L}}
\renewcommand{\H}{\mathcal{H}}
\newcommand{\Hom}{\mathcal{H}\textnormal{om}}
\newcommand{\tbH}{{\text{H}}}
\DeclareMathOperator{\tnH}{H}
\newcommand{\mR}{\mathrm{R}}
\newcommand{\IK}{\I_\mfK}

\date{\today}
\title[Resolution of the symmetric algebra of a finite base locus]{Resolution of the symmetric algebra of a finite base locus}

\author{R\'emi Bignalet-Cazalet}
\address{Universit\'e de Bourgogne Franche-Comt\'{e}, Institut de Math\'ematiques de Bourgogne,
9 avenue Alain Savary, 
BP 47870 - 21078 Dijon Cedex, France}
\email{remi.bignalet-cazalet@u-bourgogne.fr}



\begin{document}
\definecolor{biblio}{rgb}{0,0.65,1}
\definecolor{xdxdff}{rgb}{0.49,0.49,1}
\definecolor{ttttff}{rgb}{0.2,0.2,1}
\definecolor{zzzzff}{rgb}{0.6,0.6,1}

\begin{abstract}
We provide a locally free resolution of the projectivized symmetric algebra of the ideal sheaf of a zero-dimensional scheme defined by $n+1$ equations in an $n$-dimensional variety. The resolution is given in terms of the resolution of the ideal itself and of the Eagon-Northcott complex of the Koszul hull.
\end{abstract}

\maketitle

\section{Introduction}
Consider a rational map $\Phi:\P^n\dashrightarrow\P^n$ with a zero-dimensional base locus $Z$. In order to compute some invariants of $\Phi$, for instance its degree, one should resolve the indeterminacies of $\Phi$, which amounts to blow-up $Z$ or equivalently to work with the Rees algebra of the ideal $I_Z$ of $Z$ in $\P^n$. This is not quite easy in general, however a first step is to take the symmetric algebra $\tnS(I_Z)$ of $I_Z$, this is a larger algebra as the Rees algebra is obtained from it by killing the torsion part. This problem is closely related to the papers \cite{BuChJo2009TorSymAlg} and \cite{BuChSi2010ElimAndNonlinEq} about the torsion of the symmetric algebra. So a natural question is what is the shape of the resolution of $\textnormal{S}(I_Z)$, in particular, is it determined by some process involving the resolution of $I_Z$?

The goal of this paper is to give an affirmative answer to this question. Indeed, we provide a resolution of $\textnormal{S}(I_Z)$ in terms of the pulled-back resolution of the dualizing module of $Z$, up to some shift in degree, and of the Eagon-Northcott complex associated with another still larger algebra, which we call the Koszul hull.

Let us state our results more precisely, in a  geometric fashion. Fix an algebraically closed field $\k$, and let $X$ be an $n$-dimensional smooth quasi-projective variety over $\k$. Let $\mathcal{L}$ be a line bundle over $X$ and let $\tnV$ be an $(n+1)$-dimensional subspace of $\text{H}^0(X,\L)$. The image of the \emph{evaluation map} $\tnV\otimes \L^\vee \rightarrow \O_X$ is an ideal sheaf $\I_Z$ of a closed subscheme $Z$ in $X$. Given a basis $(\phi_0,\ldots,\phi_n)$ of $\tnV$, this provides a rational map $\Phi:X\dashrightarrow \P(\tnV)$ sending  $x\in X$ to $\Big(\phi_0(x):\ldots:\phi_n(x)\Big)$ and defined away from $Z$.

Let $\PI=\P_X(\IZ)$ be the projectivization of the ideal sheaf $\IZ$. The surjection $\tnV\otimes \L^\vee\rightarrow \IZ$ induces a closed embedding $\PI\hookrightarrow\PnX$. The goal of this paper is to establish a locally free resolution of $\PI$ over $\PnX$ under the assumption that $Z$ is zero-dimensional.

Let $p:\PnX\rightarrow X$ be the projective bundle map, $\xi$ be the first Chern class $c_1\Big(\O_{\PnX}(1)\Big)$ of $\O_{\PnX}(1)$ and, depending on the context, $\eta$ be either $c_1(\L)$ or the pull back $\p c_1(\L)$ of $c_1(\L)$ by $p$. Put \[ \mfQ_{i,j}=(\overset{i+1}{\wedge}\tnV)\otimes \O_{\PnX}(-(j+1)\xi-(i-j)\eta) \hspace{0.5cm}\text{for }1\leq i\leq n\text{ and } 0\leq j\leq i-1\]and $\mfQ_i=\overset{i}{\underset{j=0}{\oplus}}\mfQ_{i,j}$. The sheaves $\mfQ_i$ are the terms of the Eagon-Northcott complex associated with a map
\[\psi:\tnV\otimes \O_{\P_X^n}\rightarrow\O_{\P^n_X}(\eta)\oplus\O_{\P^n_X}(\xi).\] The complex takes the form:

\begin{equation}\label{resK}\tag{\text{$\mfQ_\bullet$}}
\begin{tikzcd}[row sep=0em,column sep=1cm,minimum width=2em]
0 \ar{r}& \mfQ_n \ar{r}& \ldots \ar{r}& \mfQ_1 \ar{r}& \O_{\PnX}\\
\end{tikzcd}
\end{equation}
see \cite[2.C]{BrunsVetter1988DetRing} for details about this construction.

Assume $\dim(Z)=0$ and let:
\begin{equation}\label{resZ}\tag{\text{$\mcP_\bullet$}}
\begin{tikzcd}[row sep=0em,column sep=0.8cm,minimum width=2em]
  0 \ar{r}& \mcP_n \ar{r}& \ldots \ar{r}& \mcP_1 \ar{r}& \mcP_0 \ar{r}&\O_{Z} \ar{r}& 0\\
\end{tikzcd}
\end{equation}
be a locally free resolution of $\O_Z$, so here $\mcP_0=\O_X$ and $\mcP_1=\tnV\otimes \L^\vee$. 

Set \[\mcP_i'=\p\mcP_{n+1-i}^\vee\otimes\O_{\PnX}(-n\eta-\xi)\hspace{0.5cm}\text{for }1\leq i\leq n+1\] and let $\J$ be the ideal of $\PI$ into $\PnX$. Our result is the following:
\begin{thm}\label{thmSubLin}
Under the assumption that $\dim(Z)=0$, $\PI$ is Cohen-Macaulay of dimension $n$ and there is a locally free resolution of $\J$ of the following form:
\begin{equation}\label{ESthmSubLin}\stepcounter{SExactes}\tag{R\theSExactes}
\begin{tikzcd}[row sep=3em,column sep=0.5cm,minimum width=2em]
0 \ar{r}& P'_{n+1} \ar{r}& \begin{matrix}
     \mfQ_{n} \\\oplus  \\\mcP'_{n}
\end{matrix} \arrow{r}& \ldots \arrow{r}& \begin{matrix}
     \mfQ_1 \\\oplus  \\\mcP_1'
\end{matrix} \arrow{r}& \I_{\PI} \ar{r}& 0.
\end{tikzcd}
\end{equation}
\end{thm}

Denoting by $y_i$ the homogeneous relative coordinates of the projective bundle $\PnX$, we make the following definition.

\begin{defi}
A complex $(\mathcal{R}_\bullet)$ over $\PnX$ is \emph{subregular} if for all $i$ the differential $\mathcal{R}_i\rightarrow \mathcal{R}_{i-1}$ is linear or constant in the $y$ variables.
\end{defi}

Note that we put no conditions on the coordinates of the base variety $X$.
 
With this definition, \Cref{thmSubLin} implies:
\begin{cor}\label{corSubReg}
The ideal $\J$ admits a subregular locally free resolution over $\PnX$.
\end{cor}

Looking back to the map $\Phi: X\dashrightarrow \P(\tnV)$, our motivation for \Cref{corSubReg} is to study the length of a subscheme obtained as zero locus of a global section of the sheaf $p_*(\O_{\PI}(1)^n)$ and relate it to the topological degree of $\Phi$, see \cite{Dolg2011ClassAlgGeo} for these definitions. \Cref{corSubReg} ensures that all higher direct image sheaves of $p_*$ vanish.

In the last section, we focus on a graded version of this result. Take $R=\k[x_0,\ldots,x_n]$ and $I_Z=(\phi_0,\ldots,\phi_n)$ an ideal generated by $n+1$ homogeneous polynomials of degree $\eta$. The ideal of the symmetric algebra of $I_Z$, denoted by $I_{\PI}$, is a bigraded homogeneous ideal of $S=R[y_0,\ldots,y_n]$. This time we consider the two complexes $(P'_{\bullet})$ and $(Q_{\bullet})$ obtained by taking the graded modules of global sections of $(\mcP'_\bullet)$ and $(\mfQ_\bullet)$. These are $S$-graded subregular complexes. Our result in this setting is the following.

\begin{thm}\label{thmSubRegMod}
Assume $I_Z$ is a graded homogeneous Cohen-Macaulay ideal of dimension $1$, then $\PI$ is Cohen-Macaulay and a minimal bigraded $S$-free resolution of $I_{\PI}$ reads:
\begin{equation}\label{ThResGrad}\stepcounter{SExactes}\tag{R\theSExactes}
\begin{tikzcd}[row sep=3em,column sep=0.5cm,minimum width=2em]
0 \ar{r}& Q''_{n} \ar{r}& \begin{matrix}
     Q''_{n-1} \\\oplus  \\ P''_{n-1}
\end{matrix} \arrow{r}& \begin{matrix}
     Q''_{n-2} \\ \oplus  \\ P''_{n-2}
\end{matrix} \arrow{r}& \ldots \arrow{r}& \begin{matrix}
     Q''_2 \\ \oplus  \\ P''_2
\end{matrix} \arrow{r}& P''_1 \ar{r}& I_{\PI} \ar{r}& 0
\end{tikzcd}
\end{equation}
where \[Q''_i=\overset{n}{\underset{j=1}{\oplus}}Q_{i,j}, \hspace{0.5cm} Q_{i,j}=S\Big(-(i-j)\eta,-j-1)^{\binom{n+1}{i+1}}, \hspace{0.5cm} P''_{i}= P_{i+1}\otimes S(\eta,-1).\]
\end{thm}

Moreover \Cref{thmSubLin} and \Cref{thmSubRegMod} are sharp in the following sense. If $\dim(Z)>0$, then the resolution of $\PI$ might not be subregular as shown in the following example. This example was explained to us by Aldo Conca.

\begin{ex}
In $\P^3$, consider the zero locus $Z$ of the ideal $I_Z=(-x_2^3x_3 + x_3^4,-x_2^4 - x_3^4,-x_1x_3^3 - x_3^4, x_2^2x_3^2 + x_3^4)$. The ideal $I_Z$ has dimension $2$ over $R=\k[x_0,\ldots,x_3]$, so $\dim(Z)=1$, and a minimal graded free resolution of $I_{\PI}$ reads: 
\begin{center}
\begin{tikzpicture}
  \matrix (m) [row sep=3em,column sep=1em,minimum width=2em]
  {
     \node(y){$0$}; &\node(z){$S(-5,-3)$}; &\node(b){$ \begin{matrix}S(-5,-2) \\ \oplus \\ S(-4,-3)^3\end{matrix}$}; &\node(c){$\begin{matrix} S(-4,-1) \\ \oplus \\ S(-3,-2)^3 \\ \oplus \\ S(-4,-2) \\ \oplus \\ S(-3,-3)  \end{matrix}$}; &\node(d){$\begin{matrix}S(-1,-1)\\ \oplus \\ S(-2,-1)^2 \\ \oplus \\ S(-3,-1)
\end{matrix}$}; &\node(e){$I_{\PI}$}; & \node(f){$0$}; \\};
  \path[-stealth]
  	(y) edge (z)
  	(z) edge (b)
    (b) edge (c)
    (c) edge (d)
    (d) edge (e)
    (e) edge (f);
\end{tikzpicture}
\end{center}
where we wrote the shift in the $y$ variables in the right position. Hence the resolution of $I_{\PI}$ is not subregular.
\end{ex}

The explicit computations given in this paper were made using Macaulay2. The corresponding codes are available on request.

\section{Local resolution of the symmetric algebra}

\subsection{Preliminaries and notation}

For the whole paper, $X$ is a smooth quasi-projective variety, where variety stands here for a reduced connected scheme of finite type. Set $n$ for the dimension of $X$. Let $\IZ=(\phi_0,\ldots,\phi_n)\subset \O_X$ be an ideal sheaf generated by $n+1$ linearly independent global sections of a line bundle $\L$ over $X$ and $\tnV=\textnormal{vect}(\phi_0,\dots,\phi_n)$.

\begin{notation}
We denote by $\P$ the projective bundle $\Proj\Big(\Sym(\O_X(-\eta)^{n+1})\Big)$ with its bundle map $p:\P\rightarrow X$ and relative homogeneous coordinates $y_0,\ldots,y_n$. Here $\Sym(\O_X(-\eta)^{n+1})$ refers to the sheafified symmetric algebra of $\O_X(-\eta)^{n+1}$ and $\P$ is a shorter notation for the relative projective space $\PnX$ in the introduction.

We let $\xi$ be the first Chern class of $\O_{\P}(1)$ and, depending on the context, $\eta$ stands either for $c_1(\L)$ or $\p c_1(\L)$.

For a subscheme $\mathbb{L}$ of $\P$ and any $x\in X$, we denote by $\mathbb{L}_x$ the scheme-theoretic fibre of $p$ restricted to $\mathbb{L}$ above $x$.

Moreover, if $\mathcal{J}$ is an ideal sheaf of a scheme $Y$, $\V(\mathcal{J})$ stands for the subscheme of $Y$ defined by $\mathcal{J}$.
\end{notation}

By definition, $\PI=\Proj\Bigl(\Sym(\IZ)\Bigr)$. Let 
\begin{equation}\label{presentI}\stepcounter{Present}\tag{P\thePresent}
\begin{tikzcd}[row sep=0.1cm,column sep=0.5cm,minimum width=2em]
\mcP_2 \ar{rd}\ar{rr}{M} &                   & \mcP_1 \ar{r}{\Phi}& \IZ \ar{r}&  0\\
                         & \E\ar{ur} \ar{dr} &              &           &   \\
0 \ar{ur}                &                   &  0           &           &
\end{tikzcd}
\end{equation}
be a locally free presentation of $\IZ$ where $\mcP_1=\tnV\otimes \O_X(-\eta)$ and $\Phi=(\phi_0 \; \ldots \; \phi_n)$. The composition of the canonical map $\tnV\otimes \O_{\P}\rightarrow \O_{\P}(\xi)$ sending $\phi_i$ to $y_i$ with the map $\p M:\p\mcP_2\rightarrow\p\mcP_1$ provides a map $\p\mcP_2\rightarrow \O_{\P}(\xi)$ as in the following diagram:
\begin{center}
\begin{tikzpicture}
  \matrix (m) [row sep=0.75cm,column sep=1cm,minimum width=2em]
  {
     \node(c){}; &\node(a12){};  &\node(a13){$\tnV\otimes \O_{\P}$};  \\
     \node(a21){$\p\mcP_2$}; &\node(a22){$\p\E $}; &\node(a23){$\O_{\P}(\xi)$};  \\};
  \path[-stealth]
  	(a22) edge (a13)
 	(a21) edge node[above]{$p^*M$} (a13)
 	(a21) edge (a22)
 	(a22) edge (a23)
 	(a13) edge (a23);
\end{tikzpicture}
\end{center}

So, by \cite[A.III.69.4]{Bourbaki2007Algebre}, $\PI$ is the zero scheme of the corresponding section of the composition map $s\in \tbH^0\Big(\P,\p\mcP_2^\vee\otimes \O_{\P}(\xi)\Big)$. Otherwise stated, the ideal sheaf $\J$ of $\PI$ into $\P$ is generated by the entries of the row matrix $\textbf{y}p^*M$ where $M$ is the presentation matrix appearing in \eqref{presentI} and $\textbf{y}$ stands for $(y_0\;\ldots \;y_n)$. We denote by $M_x$ the matrix obtained from $M$ by specializing at the point $x\in X$.

We emphasize the following remark. Since $\dim(X)=n$ and $\codim(Z,X)=n$, the local structure sheaf of a point $z\in Z$, denoted by $\O_{Z,z}$, is generated by at least $n$ independent sections of $\L$ lying in $\tnV$. The crucial point is to take care of the case where $z\in Z$ is a point at which $Z$ is not a complete intersection, i.e all the sections $\phi_0,\ldots,\phi_n$ are required to generate $\O_{Z,z}$.

\begin{lemma}\label{lemmaFibreNLCI}
Let $x\in X$ be a closed point. The scheme-theoretic fibre $\PI_x$ is:
\begin{enumerate}[label=\rm(\it{\roman*})]
\item\label{lemmaFibreNLCI1} a point if $x\notin Z$,
\item\label{lemmaFibreNLCI2} isomorphic to $\P_x^{n-1}$ if $x\in Z$ and $Z$ is a local complete intersection at $x$,
\item\label{lemmaFibreNLCI3} isomorphic to $\P_x^{n}$ if $x\in Z$ and $Z$ is not a local complete intersection at $x$.
\end{enumerate}

In general, $\PI_x$ is isomorphic to $\P_x^{n-r}$ where $r=\textnormal{rank}(M_x)$.

\end{lemma}

\begin{proof}
Since the formation of the symmetric algebra commutes with base change, the fibre $\PI_x$ is obtained by localizing $X$ at $x$ and taking $\P(\IZ\otimes \k_x)$, where $\k_x$ is the residue field of $\O_X$ at $x$.
\begin{enumerate}
\item[\ref{lemmaFibreNLCI1}] If $x\notin Z$, locally at $x$ the ideal $\IZ$ is just $\O_X$, so $p$ is an isomorphism of $\PI_x$ to $x$.
\item[\ref{lemmaFibreNLCI2},\ref{lemmaFibreNLCI3}] If $x\in Z$, since $Z$ has codimension $n$ in $X$, a subspace of $n$ independent local sections of $\L$ from $\tnV$ is needed at least to generate $\IZ$ locally around $x$. Actually such subspace exists if and only if $Z$ is a local complete intersection (LCI) at $x$. In other words, $\IZ\otimes \k_x$ is a $\k_x$-vector space which can be generated by an $n$-dimensional subspace of $\tnV$ if and only if $Z$ is LCI at $x$, so that $\IZ\otimes \k_x$ is isomorphic to $\k_x^n$ or to $\k_x^{n+1}$ depending on whether $Z$ is LCI at $x$ or not. Therefore $\P(\IZ\otimes \k_x)$ is isomorphic to $\P_x^{n-1}$ or $\P_x^{n}$ depending on whether $Z$ is LCI at $x$ or not.
\end{enumerate}
For the last statement, tensor \eqref{presentI} by $\k_x$ and observe that the kernel $\K_x$ of the surjection $\Phi_x:\tnV\rightarrow \IZ\otimes \k_x$ is a quotient of $\E\otimes \k_x$, which in turn is a quotient of $\mcP_2\otimes \k_x$. The composition of these surjections and of the inclusion $\K_x\rightarrow \tnV$ is just the matrix $M_x$, so $\ker(\Phi_x)=\textnormal{Im}(M_x)$. Therefore $\dim(\IZ\otimes \k_x)=n+1-\textnormal{rank}(M_x)$, which completes the proof. 
\end{proof}

\begin{rem}\label{remCompoSym} In our setting of a zero-dimensional scheme $Z$, by \cite[Proposition 20.6]{eisenbud1995algebra}, the set of points $z\in Z$ such that $\PI_z\simeq \P_z^n$ is equal set theoretically to $\V\Big(\Fitt_{n}(\IZ)\Big)$ where $\Fitt_{n}(\IZ)$ is the ideal generated by the entries of $M$.
\end{rem}
Now, we take the Koszul complex with respect to the map $\tnV\otimes \O_X(-\eta)=\mcP_1\xrightarrow{\Phi} \O_X$ and we write $k_i(\Phi)$ for the $i$-th differential of the Koszul complex. We have the following sequence:
\begin{equation}\stepcounter{SExactes}\tag{K\theSExactes}\label{KoszulPres}
\begin{tikzcd}[row sep=0.1cm,column sep=0.75cm,minimum width=2em]
\wedge^2\mcP_1 \arrow{dr}\arrow[rr,"{k_1(\Phi)}"]& &\mcP_1 \arrow[r,"{\Phi}"]& \IZ \ar{r}& 0\\
 & \F_1 \arrow{ur}\arrow{dr} & & & \\
 0 \arrow{ur}&    & 0 & & \\
\end{tikzcd}
\end{equation}
which is not exact since $Z$ is not empty and where we put $\F_1=\textnormal{Im}\Big(k_1(\Phi)\Big)$. By definition of the presentation and the Koszul complex, we have $\F_1\subset \E$ and $\E/\F_1 = \H_1(\IZ)$ where $\E$ is as in \eqref{presentI} and $\H_1(\IZ)$ stands for the first Koszul homology of the set $(\phi_0\;\ldots\; \phi_n)$ of generators of $\IZ$.

\subsection{Gorenstein nature of the Koszul hull}

We introduce now another subscheme of $\P$ which we call the Koszul hull of $\PI$. This subscheme contains $\PI$ and actually differs from $\PI$ by a copy of $\P^n_Z$, as we will see.

\begin{defi}
Set notation as in \eqref{KoszulPres} and let $\I_{\mfK}$ be the ideal sheaf generated by the entries in the row matrix $\textbf{y}\p k_1(\Phi)$. We call the \emph{Koszul hull}, denoted by $\mfK$, the subscheme in $\P$ defined by $\mfK=\V(\I_{\mfK})$.

\end{defi}

Now, we explain the strategy of the proof of \Cref{thmSubLin}. Via the inclusion $\F\subset\E$, we see that $\I_\mfK\subset \J$, that is $\PI\subset \mfK$. Hence we have the following short exact sequence:
\begin{center}
\begin{tikzpicture}
  \matrix (m) [row sep=0.1cm,column sep=0.75cm,minimum width=2em]
  {
     \node(a11){$0$}; &\node(a12){$\I_\mfK$}; &\node(a13){$\J$}; &\node(a14){$\J/\I_\mfK$}; & \node(a15){$0$.}; \\ };
  \path[-stealth]

    (a11) edge (a12)
    (a12) edge (a13)
    (a13) edge (a14)
    (a14) edge (a15);
\end{tikzpicture}
\end{center}

So in order to get the subregularity of the resolution of $\J$, we first show the subregularity of resolutions of $\I_\mfK$ and of $\J/\I_\mfK$ and from there, we show how we get the resolution of $\J$ by patching together these resolutions.

We start by analysing the Koszul hull more closely.

\begin{prop}\label{resKoszHull}\label{PropKozulCanonic}
We have the following properties.
\begin{enumerate}[label=\rm{\it(\roman*)}]
\item\label{PropKoszulCano1} The scheme $\mfK$ is determinantal. More precisely, $\I_\mfK$ is the ideal of the $2\times 2$ minors of the map $\tnV\otimes\O_{\P}\rightarrow \O_\P(\eta)\oplus\O_{\P}(\xi)$ defined by the matrix:\[\psi=\begin{pmatrix}
\phi_0 & \ldots & \phi_n \\
y_0 & \ldots & y_n
\end{pmatrix}.\]
\end{enumerate}
Under the assumption that $\dim_X(Z)=0$:
\begin{enumerate}[label=\rm{\it(\roman*)},resume]
\item\label{PropKoszulCano2} $\codim(\mfK,\P)=n$.
\item\label{PropKoszulCano3} A locally free resolution of $\J$ is the sheafification of the Eagon-Northcott complex. Namely, there is a long exact sequence:
\begin{equation}\tag{\ref*{resK}}\label{MinResK}
\begin{tikzcd}[row sep=3em,column sep=0.5cm,minimum width=2em]
0 \ar{r}& \mfQ_n \ar{r}& \ldots \ar{r}& \mfQ_2 \ar{r}& \mfQ_1 \ar{r}& \I_{\mfK} \ar{r}& 0
\end{tikzcd}
\end{equation}

where $\mfQ_i=\overset{i}{\underset{j=0}{\oplus}}\mfQ_{i,j}$ and \[\mfQ_{i,j}=(\overset{i+1}{\wedge}\tnV)\otimes \O_{\P}(-(j+1)\xi-(i-j)\eta) \hspace{0.5cm}\text{for }1\leq i\leq n\text{ and } 0\leq j\leq i-1. \]

\item\label{PropKoszulCano4} The scheme $\mfK$ is Gorenstein, more precisely we have:
\[ \omega_{\mfK}\simeq \p\omega_X\otimes\O_{\P}(n\eta-n\xi).
\]
\end{enumerate}

\end{prop}

\begin{proof}
\begin{itemize}
\item[\ref{PropKoszulCano1}] The morphism $k_1(\Phi)$ takes the form, \[k_1(\Phi)=\begin{pmatrix}
\phi_1 & \phi_2 & \ldots  \\
-\phi_0 & 0 & \ldots \\
0 & -\phi_0 & \ldots \\
\vdots & 0 & \ldots \\
\vdots & \vdots & \ldots 
\end{pmatrix}\] and $\IK$ is generated by the entries in the row matrix $\textbf{y}\p k_1(\Phi)$. Those entries are the same as the $2\times 2$ minors of the matrix $\psi$.
\item[\ref{PropKoszulCano2}] We argue set-theoretically by looking at the fibres of the map $\mfK \to X$ obtained as restriction of $p$ to $\mfK$. First, note that if $z \not \in Z$, then it is clear by the definition of $\mfK$ that $\mfK_z$ is a single point. On the other hand, if $z \in Z$ then $\phi_i(z)=0$ for all $i \in\lbrace 0,\ldots,n\rbrace$ so by definition of $\mfK$ we have $\mfK_z=\P^n_z$. Therefore the reduced structure of $\mfK$ is the union of $X$ and of $\cup_{ z\in Z} \P^n_z$. This proves that $\mfK$ has dimension $n$.
\item[\ref{PropKoszulCano3}] Since $\mfK$ is determinantal of the expected codimension, it is Cohen Macaulay \cite[Cor. 2.8]{BrunsVetter1988DetRing}. Hence $\depth(\IK)=\codim(\mfK,\P)=n$. Therefore the Eagon-Northcott complex provides a global resolution of the ideal $\IK$ \cite[Th. 2.16]{BrunsVetter1988DetRing}. The first map  $\wedge^2\tnV\otimes \O_{\P}\rightarrow\wedge^2\O_\P(\eta)\oplus\O_{\P}(\xi)$ of the Eagon-Northcott complex is the matrix $\wedge^2\psi$. Hence the complex \eqref{MinResK} provides a resolution of $\I_{\mfK}$.

\item[\ref{PropKoszulCano4}] By the previous item \ref{PropKoszulCano3}, a resolution of $\omega_\mfK$ is given by:
\begin{small}
\begin{center}
\begin{tikzpicture}
  \matrix (m) [row sep=0.5cm,column sep=0.6cm,minimum width=2em]
  {
     \node(y){$0$}; &\node(z){$\mfQ_1^\vee\otimes\omega_{\P}$}; &\node(a){$\ldots$}; &\node(c){$\mfQ_{n-1}^{\vee}\otimes\omega_{\P}$}; &\node(d){$\mfQ_n^{\vee}\otimes\omega_{\P}$}; &\node(e){$\omega_\mfK$}; & \node(f){$0$.}; \\};
  \path[-stealth]
  	(y) edge (z)
  	(z) edge (a)
    (a) edge (c)
    (c) edge node[above]{$M_1$} (d)
    (d) edge (e)
    (e) edge (f);
\end{tikzpicture}
\end{center}
\end{small}
Locally, we can write explicitly the matrix $M_1$ which is the transpose of the last matrix in the Eagon-Northcott complex. So $M_1$ has size $n\times (n-1)(n+1)$ and locally takes the form:
\begin{center}
\begin{tikzpicture}
    \matrix (m) [row sep=0cm,%
			 column sep=0cm,%
			 minimum width=2em,%
             left delimiter  = (,%
             right delimiter = )]
{
  \node(a11){$\phi_0$}; &\node(a12){}; &\node(a13){$\phi_n$}; &\node(a14){$0$};  &\node(a15){}; &\node(a16){}; &\node(a17){}; &\node(a18){}; &\node(a19){}; &\node(a110){}; &\node(a111){}; &\node(a112){}; &\node(a113){$0$};\\
  \node(a21){$y_0$}; &\node(a22){}; &\node(a23){$y_n$}; &\node(a24){$\phi_0$};&\node(a25){}; &\node(a26){$\phi_n$}; &\node(a27){$0$};&\node(a28){}; &\node(a29){}; &\node(a210){};&\node(a211){}; &\node(a212){}; &\node(a213){$0$};  \\
  \node(a31){$0$}; &\node(a32){}; &\node(a33){$0$}; &\node(a34){$y_0$}; &\node(a35){}; &\node(a36){$y_n$}; &\node(a37){$\phi_0$};&\node(a38){}; &\node(a39){$\phi_n$}; &\node(a310){$0$};&\node(a311){}; &\node(a312){}; &\node(a313){$0$}; \\
  \node(a41){}; &\node(a42){}; &\node(a43){}; &\node(a44){}; &\node(a45){}; &\node(a46){}; &\node(a47){}; &\node(a48){}; &\node(a49){}; &\node(a410){}; &\node(a411){}; &\node(a412){}; &\node(a413){}; \\
  \node(a51){$0$}; &\node(a52){}; &\node(a53){}; &\node(a54){}; &\node(a55){}; &\node(a56){}; &\node(a57){}; &\node(a58){}; &\node(a59){}; &\node(a510){$0$}; &\node(a511){$y_0$}; &\node(a512){}; &\node(a513){$y_n$}; \\
};

\draw[loosely dotted] (a11)-- (a13);
\draw[loosely dotted] (a14)-- (a113);
\draw[loosely dotted] (a21)-- (a23);
\draw[loosely dotted] (a24)-- (a26);
\draw[loosely dotted] (a27)-- (a213);
\draw[loosely dotted] (a31)-- (a33);
\draw[loosely dotted] (a34)-- (a36);
\draw[loosely dotted] (a37)-- (a39);
\draw[loosely dotted] (a310)-- (a313);
\draw[loosely dotted] (a41)-- (a413);
\draw[loosely dotted] (a51)-- (a510);
\draw[loosely dotted] (a511)-- (a513);
\node [left=0.5cm] at (a31.west) {$M_1=$};
\node [right=0.3cm] at (a313.east) {.};
\end{tikzpicture}
\end{center}

Consider an open cover of $X$ by a family of open subsets $\lbrace U_t \mid t\in J\rbrace $ such that $U_t\cap Z=\lbrace z_t\rbrace $. If $z\in U\subset U_t\backslash \lbrace z_t \rbrace$ for all $t$, then the restriction of $\IZ$ to $U$ is equal to $\O_U$ so that $\mfK_U=\PI_U=U$ is obviously Gorenstein, because $U$ is smooth.

Or else, if $z=z_t$ for some $t$, then $\phi_s(z)=0$ for all $s\in\lbrace 0,\ldots, n\rbrace$. In this case, since every point in $(y_0:\ldots:y_n)\in\P_z^n$ has at least one non zero coordinate, the matrix $(M_1)_z$ has corank $1$. This shows that for any point of $\PI$, the stalk of $\omega_\mfK$ has rank $1$ at that point, so $\omega_{\mfK}$ is locally free of rank one. Hence $\mfK_{U_j}$ is Gorenstein. This proves that $\mfK$ is Gorenstein.

Now, we show the isomorphism \[\omega_{\mfK}\simeq\p\omega_X\otimes \O_{\P}(n\eta-n\xi).\]

To do this, we first give an explicit formula for $\omega_{\mfK}$ by describing the scheme $\mfK$ as a complete intersection into a larger projective bundle (see \cite{Ein1993Cohomology} for more details about this construction). Let $\bbB$ be the projective bundle $\P\Big(\O_{\P}(\eta)\oplus\O_{\P}(\xi)\Big)$ and put $\zeta$ for the relative hyperplane class of the bundle map $q:\bbB\rightarrow \P$. A divisor $D$ in $|\O_{\bbB}(\zeta)|$ corresponds to a map ${\psi_D}:\O_{\P}\rightarrow\O_{\P}(\eta)\oplus \O_{\P}(\xi)$. Since the matrix $\psi$ whose $2\times 2$ minors define $\mfK$ has constant rank $1$ over $\mfK$, the map $q$ restricts to an isomorphism from the complete intersection $\cap_{i=0}^n D_i$ to $\mfK$, where $D_i$ corresponds to $\psi_{D_i}=(\phi_i,y_i)$.

Therefore, by adjunction we have: 
\begin{equation}\label{EqCano}
q^*\omega_{\mfK}\simeq \omega_{\bbB}\Big( (n+1)\zeta \Big).
\end{equation}

Next, we show that:
\begin{equation}\label{IsoRestr}
\O_{\mfK}(\zeta)\simeq \O_{\mfK}(\eta).
\end{equation}

Indeed, given a divisor $D\in|\O_{\bbB}(\zeta)|$, the intersection $D\cap \mfK$ is defined in $\P$ by the vanishing of the $2\times 2$ minors of the matrix:
\[ \begin{pmatrix}
\phi_0 & \ldots & \phi_n & \phi_D \\ y_0 & \ldots & y_n & y_D
\end{pmatrix}, \]
where $\psi_D=(\phi_D,y_D)$ corresponds to $D$. Since $y_D$ lies in $\langle y_0,\ldots,y_n\rangle$, this matrix is equivalent up to row and column operations to:
\[ \begin{pmatrix}
\phi_0 & \ldots & \phi_n & \phi_D' \\ y_0 & \ldots & y_n & 0
\end{pmatrix}, \]
for some $\phi_D'\in \tnH^0(X,\L)$.

This means that the ideal of $D\cap \mfK$ in $\mfK$ is generated by $(y_0\phi_D',\ldots, y_n\phi_D')$. Since all the $y_i$ do not vanish simultaneously, this implies that $\O_{\mfK}(\xi)$ is generated by the restriction to $\mfK$ of $\phi_D'$. Hence $\O_{\mfK}(\zeta)\simeq \O_{\mfK}(\eta)$ and we compute:
\[\omega_\P\simeq\p\omega_X\otimes \O_{\P}\Big(-(n+1)\xi\Big)\]
and therefore:
\[ \omega_{\bbB}\simeq q^*\omega_{\P}\otimes\O_{\bbB}(-2\zeta+\eta+\xi).\]

Hence by \eqref{EqCano} and \eqref{IsoRestr}, we get that $\omega_{\mfK}\simeq \p\omega_X\otimes\O_{\P}(n\eta-n\xi)$.
\end{itemize}
\end{proof}

\subsection{Description of the quotient $\I_{\PI}/\I_\mfK$} We show now the subregularity of a locally free resolution of the quotient $\I_{\PI}/\I_\mfK$.

\begin{prop}\label{resQuotient} We have the following isomorphism:\[\J/\I_\mfK\simeq \p(\omega_Z\otimes\omega_X^\vee)\otimes\O_{\P}(-n\eta-\xi).\]
\end{prop}

The proof of this proposition is the object of \Cref{quotientIdeals}. Its proof and the proof of \Cref{resQuotient} rely mostly on \cite[Theorem 21.23]{eisenbud1995algebra}. We refer to \cite{eisenbud1995algebra} for the relevant definitions.

\begin{lemma}\label{quotientIdeals}
The quotient ideal sheaf $(\I_{\mfK}:\I_{\PI})$ is isomorphic to $\p\IZ$.
\end{lemma}

\begin{proof}

As in the proof of \Cref{resKoszHull}, we denote by $k_1(\textbf{y})$ the first differential in the Koszul complex associated to the map $(y_0\;\ldots\;y_n)$. We denote also by $\IPIK$ the ideal of $\PI$ in $\mfK$ and $\bbW$ stands for the scheme $\p Z$. Of course we have $\bbW\simeq\P^n_Z$. The inclusion $\I_{\mfK}\subset \IW$ explains the right horizontal exact sequence in the following commutative diagram:
\begin{center}
\begin{tikzpicture}
  \matrix (m) [row sep=2em,column sep=3em,minimum width=2em]
  {
     \node(18){}; &\node(19){}; &\node(11){}; &\node(12){$0$}; &\node(13){}; \\
     \node(28){}; &\node(29){}; &\node(21){$\wedge^2\tnV\otimes\O_{\P}(-\xi-\eta)$}; &\node(22){$\I_\mfK$}; &\node(23){$0$}; \\
     \node(38){$0$}; &\node(39){$\p\E$}; &\node(31){$\tnV\otimes\O_{\P}(-\eta)$}; &\node(32){$\I_{\bbW}$}; &\node(33){$0$}; \\
     \node(48){$0$}; &\node(49){$\mathcal{C}$}; &\node(41){$\O_{\P}(\xi-\eta)$}; &\node(42){$\I_{\bbW}/\I_{\mfK}$}; &\node(43){$0$};\\
     \node(58){}; &\node(59){$0$}; &\node(51){$0$}; &\node(52){$0$}; &\node(53){};  \\};
  \path[-stealth]
  (21) edge node[right]{$k_1(\textbf{y})$} (31)
  (31) edge  node[right]{$\textbf{y}$}(41)
  (41) edge (51)
  
  (39) edge (49)
  (49) edge (59)
  
  (12) edge (22)
  (22) edge (32)
  (32) edge (42)
  (42) edge (52)
  (39) edge node[above]{$\p M$} (31)
  
  (21) edge  node[above]{$\textbf{y}\p k_1(\Phi)$}(22)
  (22) edge (23)
  
  (31) edge node[above]{$\p\Phi$} (32)
  (32) edge (33)
  (38) edge (39)  
  
  (48) edge (49)
  (41) edge (42)
  (42) edge (43)
  (49) edge (41)
  (39) edge[dash pattern=on 2pt off 2pt]node[below]{$\beta$}(41);
\end{tikzpicture}
\end{center}
The commutativity in the right above square comes from the following fact. Writing down the matrix $k_1(\textbf{y})$ as follows: 
\[k_1(\textbf{y})=\begin{pmatrix}
y_1 & y_2 & \ldots  \\
-y_0 & 0 & \ldots \\
0 & -y_0 & \ldots \\
\vdots & 0 & \ldots \\
\vdots & \vdots & \ldots 
\end{pmatrix}\]
and similarly for $k_1(\Phi)$, it is direct computation to show that $\textbf{y}\p k_1(\Phi)=\p\Phi k_1(\textbf{y})$.

Hence, the image of the map $\beta=\textbf{y}\p M$ is exactly the ideal $\J(\xi)$ and we have that:$$\textnormal{Ann}(\I_{\bbW}/\I_{\mfK})\simeq\J.$$
Now we use the assumption that $Z$ is zero-dimensional. Since the statement is local and the formation of the symmetric algebra commutes with base change, we can assume that $\O_{\P}$ and $\O_{\mfK}$ are Gorenstein local rings. We apply \cite[Theorem 21.23.a.]{eisenbud1995algebra} to the Gorenstein scheme $\mfK$ and to the ideal sheaf $\IPIK$. 

We denote by $\I_{\bbW,\mfK}$ the ideal of $\bbW$ in $\mfK$. Since $\bbW$ has codimension $0$ in $\mfK$ and has no embedded components, the ideals $\I_{\bbW,\mfK}$ and $\IPIK$ are linked in $\O_\mfK$. This shows that $\I_{\bbW,\mfK}=\textnormal{Ann}(\IPIK)$. Now, since we have already $\I_{\mfK}\subset \IW$, the equality occurs as ideal sheaves of $\O_{\P}$ itself. Moreover we have the isomorphism $\textnormal{Ann}(\IPIK)\simeq(\I_{\mfK}:\J)$. Hence:
\[\IW=\p\IZ\simeq(\I_{\mfK}:\J).\]
\end{proof}

\begin{proof}[Proof of \Cref{resQuotient}]
As above, we can assume that $\O_{\P}$ and $\O_{\mfK}$ are Gorenstein local rings and we apply \cite[Theorem 21.23]{eisenbud1995algebra} to $\O_{\mfK}$. We denote again by $\IPIK$ the ideal of $\PI$ in $\mfK$ and by $\I_{\bbW,\mfK}$ the ideal of $\bbW$ in $\mfK$ (recall that $\bbW=\p Z$).

Since $\IPIK$ has codimension $0$ in $\O_{\mfK}$, we have that $(\I_{\mfK}:\J)$ and $\I_{\PI,\mfK}$ are linked. But following the notation in \Cref{quotientIdeals}, $(\I_{\mfK}:\I_{\PI})\simeq \I_{\bbW,\mfK}$.

Moreover, $\bbW$ is Cohen-Macaulay as a pull back of $Z$ so $\PI$ is also Cohen-Macaulay and we have: 
\[ \I_{\bbW,\mfK}\simeq\omega_{\PI}\]
where $\omega_{\PI}$ is the canonical sheaf of $\PI$. Summing up, we have that:
 
\[ \omega_{\bbW}\otimes\omega_{\mfK}^\vee\simeq\I_{\PI,\mfK}\simeq \J/\I_{\mfK}. \]

Now, since $\bbW\simeq \P^n_Z$, we have $\omega_{\bbW}\simeq\p\omega_Z\otimes\O_{\P}(-(n+1)\xi)$. Therefore, by \Cref{resKoszHull}:
$$\J/\I_{\mfK}\simeq \omega_{\bbW}\otimes\omega_{\mfK}^\vee\simeq \p(\omega_Z\otimes\omega_X^\vee)\otimes\O_{\P}(-n\eta-\xi).$$
\end{proof}

Denoting $\H_1(\IZ)$ the first Koszul homology associated to $\Phi:\tnV\otimes\O_X(-\eta)\rightarrow\O_X$, as in \eqref{KoszulPres}, we emphasize the following point in order to elucidate the nature of the sheaf $\J/\I_{\mfK}$.
\begin{prop}\label{Homology}
The sheaf $\J/\I_{\mfK}$ is isomorphic to the pull-back of the first homology $\H_1(\IZ)$ of $\Phi$ up to a shift. More precisely, we have \[\J/\I_{\mfK}\simeq\p\H_1(\IZ)\otimes\O_{\P}(\eta-\xi).\]
\end{prop}

\begin{proof}
To shorten the notation, we set $\H_1$ for $\H_1(\IZ)$. We are going to show that \begin{equation}\label{isoKoszul}
\H_1\simeq\omega_Z\otimes \omega_X^\vee\Bigl(-(n+1)\eta\Bigr).
\end{equation}

First, $\omega_Z\simeq\Ext^n(\O_Z,\omega_X)$. Hence, we will prove \eqref{isoKoszul} by showing that \[\O_Z\simeq \Ext^n(\H_1,\omega_X)\otimes\omega_X^\vee\Bigl( -(n+1)\eta\Bigr).\]

To this end, let:
\begin{equation*}\label{KoszulComplex}\stepcounter{SExactes}\tag{K\theSExactes}
\begin{tikzcd}[row sep=0.1cm,column sep=0.4cm,minimum width=2em]
0 \ar{r}& \overset{n+1}{\wedge}\mcP_1 \ar{r}{k_n(\Phi)}& \ldots \ar{dr}\ar{rr}&      & \overset{2}{\wedge}\mcP_1\ar{dr}\ar{rr}{k_1(\Phi)}&  &\mcP_1 \ar{r}{\Phi}& \IZ \ar{r}& 0 \\
        &               &             &  \F_2 \ar{ur} &      &  \F_1\subset{\E}\ar{ur} &            &
\end{tikzcd}
\end{equation*}
be the Koszul complex associated with $\Phi=(\phi_0\;\ldots\;\phi_n)$, where $\overset{i}{\wedge}\mcP_1=(\wedge^{i}\tnV)\otimes \O_X(-i\eta)$. Since $\codim(Z,X) = \depth(\I_Z)=n$ the Koszul homology is concentrated in degree $1$ and by definition $\H_1=\E/\F_1$.

Applying the functor $\Hom(-,\omega_X)$ to \eqref{KoszulComplex}, we obtain:
\begin{small}
\begin{equation*}
\begin{tikzcd}[row sep=0.1cm,column sep=0.3cm,minimum width=2em]
 0  \ar{r}& \Hom(\F_1,\omega_X) \ar{r}& \ldots \ar{r}&\tnV\otimes \omega_X(n\eta)  \ar{r}& \omega_X\Bigl((n+1)\eta\Bigr) \ar{r}&\Ext^1(\F_{n-1},\omega_X)\ar{r}& 0
\end{tikzcd}
\end{equation*}
\end{small}
and it is a computation to show that $\Ext^1(\F_{n-1},\omega_X)\simeq \Ext^{n-1}(\F_{1},\omega_X)$.

The last point is that $\Ext^{n-1}(\F_{1},\omega_X) \simeq \Ext^n(\H_1,\omega_X)$. Indeed, by the long exact sequence associated to the short exact sequence:
\begin{equation*}
\begin{tikzcd}[row sep=0.1cm,column sep=0.4cm,minimum width=2em]
  0  \ar{r}&  \F_1  \ar{r}& \E \ar{r}& \H_1 \ar{r}&0
\end{tikzcd}
\end{equation*}

we have the following exact sequence:
\begin{equation*}
\begin{tikzcd}[row sep=0.1cm,column sep=0.4cm,minimum width=2em]
\Ext^{n-1}(\E,\omega_X) \ar{r}& \Ext^{n-1}(\F_1,\omega_X) \ar{r}& \Ext^{n}(\H_1,\omega_X) \ar{r}& \Ext^{n}(\E,\omega_X) 
\end{tikzcd}
\end{equation*}
and $\Ext^{n-1}(\E,\omega_X) =\Ext^{n}(\E,\omega_X)=0$ since $Z$ is locally Cohen-Macaulay.

Moreover, the last map $k_n(\Phi)$ of the Koszul complex is the transpose of the first map $\Phi$ up to signs. Thus the maps in the sequence: \[\tnV\otimes\omega_X(n\eta)\xrightarrow{}\omega_X\Bigl((n+1)\eta\Bigr)\rightarrow \Ext^n(\H_1,\omega_X)\rightarrow 0\]
are the same as the maps in the exact sequence:
\[\mcP_1\xrightarrow{\Phi} \O_X\rightarrow \O_Z\rightarrow 0. \]
Taking care of the twisting, this means that $\O_Z\otimes\omega_X\Bigl((n+1)\eta\Bigr)\simeq \Ext^n(\H_1,\omega_X)$.

This implies $\H_1\simeq\omega_Z\otimes \omega_X^\vee\Bigl(-(n+1)\eta\Bigr)$.
\end{proof}

\begin{rem}
To enlighten the construction of the sheaves $\mcP_i'$ for $i\in\lbrace 1,\ldots, n+1\rbrace$ in the following proof of \Cref{thmSubLin}, recall that the complex:
\begin{equation}\tag{\ref*{resGradZ}}
\begin{tikzcd}[row sep=0em,column sep=0.5cm,minimum width=2em]
0 \ar{r}& \mcP_n  \ar{r}& \ldots \ar{r}& \mcP_1 \ar{r}& \mcP_0 \ar{r}& \O_{Z} \ar{r}& 0
\end{tikzcd}
\end{equation}
is a locally free resolution of $\O_Z$. Hence, a locally free resolution of $\omega_Z$ reads :
\[
\begin{tikzcd}[row sep=0em,column sep=0.5cm,minimum width=2em]
0 \ar{r}& \mcP_0^\vee\otimes\omega_X  \ar{r}& \ldots \ar{r}& \mcP_n^\vee\otimes\omega_X \ar{r}&\omega_{Z} \ar{r}& 0
\end{tikzcd}
\]
from which we can read a locally free resolution of $\omega_Z\otimes \omega_X^\vee$.
\end{rem}

\begin{proof}[Proof of \Cref{thmSubLin}]
As we saw in \Cref{lemmaFibreNLCI} and in the proof of \Cref{resQuotient}, $\PI$ is Cohen-Macaulay of dimension $n$.

Moreover, by \Cref{resKoszHull} and \Cref{resQuotient}, we have the following commutative diagram:
\begin{center}
\begin{tikzpicture}
  \matrix (m) [row sep=1em,column sep=1.5em,minimum width=2em]
  {
\node(){}; &\node(){}; &\node(){}; &\node(){}; &\node(){}; &\node(){}; &\node(z){$0$}; &\node(){}; \\

\node(){}; &\node(a1){$0$}; &\node(a){$\mfQ_n$}; &\node(b){$\ldots$}; &\node(c){$\mfQ_2$}; &\node(d){$\mfQ_1$}; &\node(e){$\I_{\mfK}$}; &\node(f){$0$}; \\

\node(){}; &\node(g1){}; &\node(g){}; &\node(h){}; &\node(i){}; &\node(j){}; &\node(k){$\I_{\PI}$}; & \node(l){}; \\

\node(s0){$0$}; &\node(s1){$\mcP_{n+1}'$}; &\node(s){$\mcP_n'$}; &\node(n){$\ldots$}; &\node(o){$\mcP_2'$}; &\node(p){$\mcP_1'$}; &\node(q){$\J/\I_\mfK$}; & \node(r){$0$.}; \\

\node(){}; &\node(){}; &\node(){}; &\node(){}; &\node(){}; &\node(){}; &\node(y){$0$}; & \node(){}; \\
     };
  \path[-stealth]
  	(s0) edge (s1)
  	(s1) edge (s)
    (a1) edge (a)
    (s1) edge (s)
  	(a) edge (b)
  	(b) edge (c)
  	(c) edge (d)
  	(d) edge (e)
  	(e) edge (f)
  	(e) edge (k)
  	(k) edge (q)
  	(s) edge (n)
  	(n) edge (o)
  	(o) edge (p)
  	(p) edge (q)
  	(q) edge (r)
  	(z) edge (e)
  	(q) edge (y);
\end{tikzpicture}
\end{center}
where
\[ \mfQ_i=\overset{i-1}{\underset{j=0}{\oplus}}\Big((\overset{i+1}{\wedge}\tnV)\otimes \O_{\P}(-(j+1)\xi-(i-j)\eta)\Big)\]
and
\[\mcP_i'=\p\mcP_{n+1-i}^\vee\otimes\O_{\P}(-n\eta-\xi)\hspace{0.5cm}\text{for }1\leq i\leq n+1.\]

To show that these resolutions patch together to give the desired resolution of $\J$, it suffices to prove that $\textnormal{Ext}^1\Bigl(\mcP_1',\I_\mfK\Bigr)=0$ that is $\textnormal{H}^1\Bigr(\P,\I_{\mfK}\otimes \mcP_1'^\vee\Bigr)=0 $.

Hence it suffices that $\textnormal{H}^i\left(\P,\mfQ_i \otimes \mcP_1'^\vee\right)=0$ for all $i\in \lbrace 1,\ldots, n\rbrace$. Kunneth formula implies these vanishings since the cohomology groups $\tnH^i\Big(\Pnk,\O_{\Pnk}(-j)\Big)$ vanish for all $j=0,\ldots, i-1$. In the case $i=n$, we use that \[\tnH^n\Big(\Pnk,\O_{\Pnk}(-j)\Big)\simeq \tnH^0\Big(\Pnk,\O_{\Pnk}(j-n-1)\Big)\] and the fact that $ j-n-1\leq -2$.

This shows eventually \Cref{thmSubLin}.
\end{proof}

We summarize \Cref{thmSubLin} into the following corollary.
\begin{cor}
Under the assumption that $\dim(Z)=0$, the ideal $\I_{\PI}$ has a resolution of the following form:
\begin{equation*}
\begin{tikzcd}[row sep=3em,column sep=1em,minimum width=2em]
0 \ar{r}& \mathcal{G}_{n+1} \ar{r}& \mathcal{G}_n \ar{r}& \ldots \ar{r}& \mathcal{G}_2 \ar{r}& \mathcal{G}_1 \ar{r}& \I_{\PI} \ar{r}& 0
\end{tikzcd}
\end{equation*}
where $\mathcal{G}_i=\underset{j=1}{\overset{i}{\oplus}}\p\mathcal{T}_{ij}\otimes \O_{\P}(-j\xi)$ when $i\in\lbrace 1,\ldots,n\rbrace$ and $\mathcal{G}_{n+1}= \p\mathcal{T}_{n}\otimes \O_{\P}(-\xi)$ for some locally free sheaves $\mathcal{T}_{ij}$ and $\mathcal{T}_{n}$ over $X$.
\end{cor}
\section{Graded free resolution of the symmetric algebra}
Now, we turn to the analysis of a resolution of the symmetric algebra of a homogeneous ideal of the polynomial ring $R=\k[x_0,\ldots,x_n]$. So let $I_Z=(\phi_0,\ldots,\phi_n)\subset R$ be an ideal generated by $n+1$ linearly independent homogeneous polynomials each one of the same degree $\eta\geq 2$. We will denote by $R_Z$ the quotient $R/I_Z$ and by $Z$ the subscheme $\V(I_Z)$ of $\Pnk$.

We will assume that $\dim(Z)=0$ and that $R_Z$ is a graded Cohen-Macaulay ring.

As above let:
\begin{equation}\label{resGradZ}\tag{\text{$P_{\bullet}$}}
\begin{tikzcd}[row sep=3em,column sep=0.5cm,minimum width=2em]
0 \ar{r}& P_n \ar{r}& \ldots \ar{r}& P_2 \arrow[r, "M"]& P_1 \ar{r}& I_Z \ar{r}& 0
\end{tikzcd}
\end{equation}
be a minimal graded free resolution of $I_Z$, $M$ being the presentation matrix of $I_Z$ and $P_1=R(-\eta)^{n+1}$.

As in the previous section, let $k_1(\Phi):\wedge^2 P_1\rightarrow P_1$ be the second differential of the Koszul complex associated with the map $\Phi:P_1\xrightarrow{(\phi_0\;\ldots \;\phi_n)}R$. Put $F=\textnormal{Im}\Big(k_1(\Phi)\Big)$ in order to have the following exact sequence:
\begin{center}
\begin{tikzpicture}
  \matrix (m) [row sep=0.1cm,column sep=0.75cm,minimum width=2em]
  {
     \node(a11){$R(-2\eta)^{\binom{n+1}{2}}$}; &\node(a12){}; &\node(d){$R(-\eta)^{n+1}$}; &\node(e){$I_Z$}; & \node(f){$0$.}; \\
     \node(a21){}; &\node(a22){$F $}; &\node(a23){}; &\node(){}; & \node(){}; \\
     
     \node(a31){$0$}; &\node(a32){}; &\node(a33){$0$}; &\node(){}; & \node(){}; \\};
  \path[-stealth]
  	(a31) edge (a22)
  	(a22) edge (a33)
 	(a11) edge (a22)
 	(a22) edge (d)
    (a11) edge node[above] {$k_1(\Phi)$}(d)
    (d) edge node[above]{$\Phi$} (e)
    (e) edge (f);
\end{tikzpicture}
\end{center}
\begin{defi}
Set $S=R[y_0,\ldots,y_n]$ and $\textbf{y}=(y_0\;\dots\;y_n)$. We let $I_{\PI}$ be the ideal of $S$ generated by the entries in the row matrix $\textbf{y}M$ and $I_\mfK$ be the ideal of $S$ generated by the entries in the row matrix $\textbf{y} k_1(\phi)$.
\end{defi}

Here, as above, $F\subset E$ so $I_\mfK\subset I_{\PI}$.

\begin{notation}
Since $S$ is bigraded by the variables $\textbf{x}$ and $\textbf{y}$, $S(-a,-b)$ stands for a shift in $\textbf{x}$ for the left part and $\textbf{y}$ for the right part.

As above, we denote by $\P$ the product $\Pnk\times \Pnk$ and by $p:\P^n\times\P^n\rightarrow \P^n$ the first projection.
\end{notation}

To show \Cref{thmSubRegMod}, the strategy is initially the same as in the previous section, but since we are dealing with free resolutions, the resolutions of $I_\mfK$ and $I_{\PI}/I_\mfK$ will patch together providing a resolution of $I_{\PI}$ without further checking. We will explain afterwards how we deduce from this resolution a minimal bigraded free resolution of $I_{\PI}$.

\subsection{The Koszul hull}
All the arguments of the proof of \Cref{resKoszHull} remain valid in the graded homogeneous setting. So the ideal $I_\mfK$ has the following properties:
\begin{enumerate}[label=\rm{\it(\roman*)}]
\item\label{KoszulGrad1} $I_\mfK$ is a determinantal ideal.
\end{enumerate}
Under the assumption that $\codim(Z,\Pnk)=n$:
\begin{enumerate}[resume,label=\rm{\it(\roman*)}]
\item\label{KoszulGrad2} $\codim(\mfK,\P)=n$.
\item\label{KoszulGrad3} a graded free resolution of $I_\mfK$ is the Eagon-Northcott complex associated to the matrix:
\[\psi=\begin{pmatrix}
\phi_0 & \ldots & \phi_n \\
y_0 & \ldots & y_n
\end{pmatrix}.\]
Hence, the following complex is a bigraded free resolution of $I_\mfK$:
\begin{equation}\stepcounter{SExactes}\tag{\text{$Q_\bullet$}}\label{GrMinResK}
\begin{tikzcd}[row sep=3em,column sep=1.5em,minimum width=2em]
0 \ar{r}& Q_n \ar{r}& \ldots \ar{r}& Q_2 \ar{r}& Q_1 \ar{r}& I_{\mfK} \ar{r}& 0
\end{tikzcd}
\end{equation}
where $Q_i=\overset{i}{\underset{j=0}{\oplus}}Q_{i,j}$ and \[Q_{i,j}=S\Big(-(i-j)\eta,-j-1)^{\binom{n+1}{i+1}} \hspace{0.5cm}\text{for }1\leq i\leq n\text{ and } 0\leq j\leq i-1. \]
\item\label{KoszulGrad4} The scheme $\mfK$ is Gorenstein, more precisely the canonical module $\omega_{S_{\mfK}}$ of $\mfK$ verifies: \[\omega_{S_{\mfK}}\simeq S\Big( n(\eta-1)-1,-n\Big).\]
\end{enumerate}

\subsection{Identification of the quotient $I_{\PI}/I_\mfK$}

We denote by $\omega_{R_Z}$ the canonical module of $Z$. All the arguments of \Cref{PropKozulCanonic} and \cite[Theorem 21.23]{eisenbud1995algebra} apply in the graded case since $R_Z$ is a graded Cohen-Macaulay ring of depth $n$. Hence we have that:
\[ I_{\PI}/I_{\mfK}\simeq \omega_{R_Z}\otimes S(n(1-\eta)+1,-1)\text{ as }S\text{-modules}.\]

Recall that \eqref{resGradZ} is a minimal graded free resolution of $I_Z$. Put \[P_i'= P_{n+1-i}^\vee\otimes S\Big(-n\eta,-1) \hspace{0.5cm}\text{for }i\in \lbrace 1,\ldots,n+1\rbrace . \]

Then the complex:
\begin{equation*}\label{resGradNonMin1}\tag{\ref*{ThResGrad}'}
\begin{tikzcd}[row sep=3em,column sep=0.5cm,minimum width=2em]
0 \ar{r}& P_{n+1}' \ar{r}& \begin{matrix}
     Q_n \\\oplus  \\ P_n'
\end{matrix} \arrow{r} &\ldots \arrow{r} &\begin{matrix}
     Q_2 \\\oplus  \\ P_2'
\end{matrix} \arrow{r} &\begin{matrix}
     Q_1 \\\oplus  \\ P_1'
\end{matrix} \arrow{r}& I_{\PI} \ar{r}& 0
\end{tikzcd}
\end{equation*}
is a bigraded free resolution of $I_{\PI}$.

\subsection{Homotopy of complexes}
We turn now to the problem of extracting a minimal bigraded free resolution of $I_{\PI}$ from \eqref{resGradNonMin1}. In order to do so, we show first the following result.

\begin{prop}\label{pfIX}
There is a canonical isomorphism
\[p_{*}\O_{\PI}(\xi)\simeq \I_Z\]
where $\O_{\PI}(\xi)$ and $\I_Z$ are the sheafification of respectively $S(0,1)$ and $I_Z$.
\end{prop}

We emphasize that this is not completely straight forward since $\PI$ is the Proj of $\I_Z$ which is not locally free (see Stack project, 26.21. Projective bundles, \href{https://stacks.math.columbia.edu/tag/01OA}{example 26.21.2}).

\begin{proof}
Since $\O_{\P}(\xi)$ is the relative ample line bundle of the projective bundle $\P=\P\Big(\O_X(-\eta)^{n+1}\Big)$, we have:
\[
\mR^kp_{*}\O_{\P}(l\eta -j\xi)=\begin{cases} 0 &\text{ for }l>0\text{ and } j\leq 0,\\
										0 & \text{ for }j\in\lbrace 1,\ldots,k-1\rbrace\text{ and any }l,\\
										\O_X(l\eta)&\text{ for }k=0\text{ and } j= 0,\\
										\O_X^{n+1}\Big((l-1)\eta\Big)&\text{ for }k=0\text{ and } j=-1.
\end{cases}\]

Therefore, applying $p_*$ to the resolution \eqref{ESthmSubLin} and chasing cohomology we get $\mR^1p_{*}\I_{\PI}(\xi)=0$. 

Recall that we denote by $\E$ the kernel of $\Phi:\O_{X}(-\eta)^{n+1}\rightarrow\I_Z$ and that $\I_{\PI}(\xi)$ is the image of the map $\p\E\rightarrow \O_{\P}(\xi)$. Let $\H$ be the kernel of this surjection and write the exact sequence:
\begin{equation*}
\begin{tikzcd}[row sep=3em,column sep=1em,minimum width=2em]
0 \ar{r}& \H \ar{r}& \p \E \ar{r}& \I_{\PI}(\xi) \ar{r}& 0.
\end{tikzcd}
\end{equation*}

Since $p_*\p\E\simeq \E$ and $\mR^1p_*\p\E=0$, applying $p_*$ to this exact sequence, we get:
\begin{equation}\label{pfX1}\tag{a}
\begin{tikzcd}[row sep=3em,column sep=1em,minimum width=2em]
0 \ar{r}& p_*\H \ar{r}& \E \ar{r}& p_*\I_{\PI}(\xi) \ar{r}& \mR^1 p_*\H \ar{r}& 0.
\end{tikzcd}
\end{equation}

Also, since we proved that $\mR^1p_*\I_{\PI}(\xi)=0$, applying $p_*$ to the canonical exact sequence
\begin{equation*}
\begin{tikzcd}[row sep=3em,column sep=1em,minimum width=2em]
0 \ar{r}& \I_{\PI}(\xi) \ar{r}& \O_{\P}(\xi) \ar{r}& \O_{\PI}(\xi) \ar{r}& 0
\end{tikzcd}
\end{equation*}
we get
\begin{equation*}\label{pfX2}\tag{b}
\begin{tikzcd}[row sep=3em,column sep=1em,minimum width=2em]
0 \ar{r}& p_*\I_{\PI}(\xi) \ar{r}& \O_X(-\eta)^{n+1} \ar{r}& p_*\O_{\PI}(\xi) \ar{r}& 0.
\end{tikzcd}
\end{equation*}

The exact sequences \eqref{pfX1} and \eqref{pfX2} fit into the following commutative diagram:
\begin{equation*}
\begin{tikzcd}[row sep=0.8em,column sep=1em,minimum width=2em]
 & 0 \ar{d}&  &  & & \\
 & p_{*}\H \ar{d}& & 0 \ar{d}& & \\
& \E \ar{d}\arrow[rr, "\simeq"]&  & \E \ar{d}& &\\
0 \ar{r}& p_{*}\I_{\PI}(\xi) \ar{rr}\ar{d}& & \O_X(-\eta)^{n+1} \ar{r}\ar{d}& p_{*}\O_{\PI}(\xi) \ar{r}\arrow[d,phantom,"{\rotatebox{90}{=}}"]& 0  \\
0 \ar{r} &\mR^1p_{*}\H \ar{rr}\ar{d}&  & \I_Z \ar{r}\ar{d}& p_{*}\O_{\PI}(\xi) \ar{r}& 0 \\
 & 0 &  & 0 &   &  \\
\end{tikzcd}
\end{equation*}
where \eqref{pfX1} is the left column, \eqref{pfX2} is the central row and the map $\I_Z\rightarrow p_*\O_{\PI}(\xi)$ in the bottom row is the canonical morphism associated to the projectivization of $\I_Z$. This morphism is an isomorphism at $X\backslash Z$ and therefore $\I_Z\rightarrow p_*\O_{\PI}(\xi)$ is injective because $\I_Z$ is torsion free. Hence $p_*\H\simeq 0\simeq \mR^1p_*\H$ and $p_{*}\O_{\PI}(\xi)\simeq \I_Z$.

\end{proof}

\begin{proof}[Proof of \Cref{thmSubRegMod}]
We work as in the previous proposition. Taking the pushforward by $p$ of the resolution of $\O_{\PI}(\xi)$ given by \eqref{ESthmSubLin} and considering the associated $R$-modules of global sections, we obtain the following graded free resolution of $I_Z$:

\begin{center}
\begin{tikzpicture}
  \matrix(m1)[row sep=0.1cm,column sep=0.5cm,minimum width=2em]
  {
     \node(a11){$0$}; &\node(a12){$P_0^\vee\Big(-(n+1)\eta\Big) $}; &\node(a13){$\begin{matrix}
    R\Big(-(n+1)\eta\Big) \\ \oplus  \\ P_0^\vee\Big(-(n+1)\eta\Big)
\end{matrix}$}; &\node(a14){$\ldots$}; &\node(){}; &\node(){}; \\};
  \path[-stealth]
    (a11) edge (a12)
    (a12) edge (a13)
    (a13) edge (a14);
\end{tikzpicture}

\begin{tikzpicture}
\matrix(m2)[row sep=0.1cm,column sep=0.5cm,minimum width=2em]
  { \node(a22){}; &\node(){}; &\node(a24){}; &\node(a25){$\ldots$}; &\node(a26){$\begin{matrix}
     R(-2\eta)^{\binom{n+1}{2}} \\\oplus  \\ P_n^\vee(-(n+1)\eta)
\end{matrix}$}; &\node(a27){$R(-\eta)^{n+1}$}; &\node(a28){$I_{Z}$}; &\node(a29){$0$.}; \\};
  \path[-stealth]
    (a25) edge (a26)
    (a26) edge (a27)
    (a27) edge (a28)
    (a28) edge (a29);
\end{tikzpicture}
\end{center}
 
This resolution is homotopic to the minimal free resolution \eqref{resGradZ} of $I_Z$. Therefore, the truncated complex $(P_{\geq 1})$ of \eqref{resGradZ} is homotopic as $S$-complex to:
\begin{equation*}
\begin{tikzcd}[row sep=3em,column sep=0.5cm,minimum width=2em]
0 \ar{r}& P'_{n-1} \ar{r}&  \begin{matrix}
Q_{n,0}\\ \oplus \\P_n'
\end{matrix} \ar{r}& \ldots \arrow[r]& \begin{matrix}
Q_{1,0}\\ \oplus \\P_1'
\end{matrix}.
\end{tikzcd}
\end{equation*}

Hence, \eqref{resGradNonMin1} is homotopic to:
\begin{equation}\label{ThResGradBis}\tag{\ref*{ThResGrad}}
\begin{tikzcd}[row sep=3em,column sep=0.5cm,minimum width=2em]
0 \ar{r}& Q''_{n} \ar{r}& \begin{matrix}
     Q''_{n-1} \\\oplus  \\ P''_{n-1}
\end{matrix} \arrow{r}& \begin{matrix}
     Q''_{n-2} \\ \oplus  \\ P''_{n-2}
\end{matrix} \arrow{r}& \ldots \arrow{r}& \begin{matrix}
     Q''_2 \\ \oplus  \\ P''_2
\end{matrix} \arrow{r}& P''_1 \ar{r}& I_{\PI} \ar{r}& 0
\end{tikzcd}
\end{equation}
where \[Q''_i=\overset{n}{\underset{j=1}{\oplus}}Q_{i,j}, \hspace{0.5cm} Q_{i,j}=S\Big(-(i-j)\eta,-j-1)^{\binom{n+1}{i+1}}, \hspace{0.5cm} P''_{i}= P_{i+1}\otimes S(\eta,-1).\]

The complex \eqref{ThResGradBis} is thus a bigraded free resolution of $I_{\PI}$.

To finish the proof of \Cref{thmSubRegMod}, it remains to show that \eqref{ThResGradBis} is minimal. This follows from the minimality of \eqref{resGradZ} and the fact that, if $i\neq i'$, there is no bigraded homogeneous piece of the same degree among $Q''_i$ and $Q_{i'}''$ or $P''_j$ for any $j\in\lbrace 1,\ldots, n-1\rbrace$.
\end{proof}

\bibliographystyle{alpha}

\end{document}